\documentclass[11pt]{amsart}
\usepackage{amssymb, amsthm, amsmath, bbm}
\usepackage{tensor}
\usepackage{mathtools}
\usepackage{mathrsfs}
\usepackage{hyperref, manfnt}
\usepackage{setspace}
\usepackage{graphicx}
\usepackage{enumitem}
\usepackage{MnSymbol}
\usepackage{fullpage}
\usepackage{tikz}
\usepackage{tikz-cd}
\usepackage{wrapfig}
\usepackage{float}
\makeatletter
\newcommand*\bigcdot{\mathpalette\bigcdot@{.5}}
\newcommand*\bigcdot@[2]{\mathbin{\vcenter{\hbox{\scalebox{#2}{$\m@th#1\bullet$}}}}}

\makeatother

\usepackage{kantlipsum}
\newcommand\Z{\mathbb{Z}}
\usepackage{bm}
\usepackage[cmtip,all]{xy}

\newtheorem*{theorem*}{Theorem}
\newtheorem{innercustomconj}{Conjecture}
\newenvironment{conj}[1]
  {\renewcommand\theinnercustomconj{#1}\innercustomconj}
  {\endinnercustomthm}

\newtheorem{theorem}{Theorem}
\newtheorem{definition}{Definition}[section]
\usepackage{listings}
\numberwithin{theorem}{section}

\newtheorem{corollary}{Corollary}[section]
\newtheorem{def-lemma}{Definition-Lemma}
\newtheorem{rem}{Remark}
\newtheorem*{cor*}{Corollary}
\newtheorem{lemma}{Lemma}

\newtheorem*{rem*}{Remark}
\usepackage{verbatim}
\numberwithin{lemma}{section}
\newtheorem*{lemma*}{Lemma}

\newtheorem*{claim*}{Claim}

\tikzset{
	symbol/.style={
		draw=none,
		every to/.append style={
			edge node={node [sloped, allow upside down, auto=false]{$#1$}}}
	}
}

\theoremstyle{definition}

\newtheorem{innercustomthm}{Theorem}
\newenvironment{customthm}[1]
  {\renewcommand\theinnercustomthm{#1}\innercustomthm}
  {\endinnercustomthm}

\newcommand{\p}{\mathbb{P}}
\newcommand{\ri}{\rightarrow}
\newcommand{\J}{\mathscr{J}}
\newcommand{\jac}{\mathcal{J}}
\newcommand{\mo}{\mathcal{O}}
\newcommand{\K}{\mathbb{C}}

\newcommand{\s}{\sigma}
\newcommand{\m}{\mathcal{M}}

\newcommand{\h}{\mathcal{H}}
\newcommand{\U}{\mathscr{U}}
\newcommand{\ve}{\mathcal{V}ec}
\newcommand{\V}{\mathcal{V}}

\newcommand{\E}{\mathscr{E}}
\newcommand{\e}{\mathcal{E}}
\newcommand{\Q}{\mathbb{Q}}
\newcommand{\C}{\mathscr{C}}
\newcommand{\mC}{\mathcal{C}}
\newcommand{\ka}{\kappa}

\newcommand{\T}{\Theta}

\newcommand{\lb}{\mathcal{L}}

\newcommand{\ext}{Filt}

\newcommand{\la}{\lambda}
\newcommand{\hr}{\hookrightarrow}

\usepackage{stackengine}
\newcommand{\hollowslash}{\setbox0=\hbox{/}\def\holwd{3pt}%
  \stackengine{-.3pt}{/}{\rlap{\kern.5pt\rule{\holwd}{0.4pt}}}{O}{r}{F}{F}{S}%
  \kern\dimexpr\holwd-\wd0-.2pt\relax%
  \stackengine{-.4pt}{/}{\llap{\rule{\holwd}{0.4pt}\kern-0.1pt}}{U}{l}{F}{F}{S}%
}

\lstdefinelanguage{Macaulay2}{
    keywords={ideal, module, matrix, ring, quotient, basis, resolution, degree, gens, hilbertSeries, dim, codim, degree, substitute, syz, for, if, then, else, do, while, return, local, try, catch, break, continue, load, needs, newClass, newPackage, peek, toString, toExternalString, net, assert, code, installPackage, viewHelp, documentation, exit, quit, clearAll},
    keywordstyle=\color{blue}\bfseries,
    sensitive=true,
    morecomment=[l]--,
    morecomment=[n]{/*}{*/},
    morestring=[b]",
    basicstyle=\ttfamily\small,
    showstringspaces=false,
    breaklines=true,
    frame=single,
    numbers=left,
    numberstyle=\tiny\color{gray},
}

\title{\MakeUppercase{Cycles in the universal moduli stack of bundles of rank two over genus two curves} }
\author{Shubham Saha}
\address{Department of Mathematics, University of California, San Diego, La Jolla, CA 92093, USA}
\email{shsaha@ucsd.edu}
\usepackage{microtype}
\begin{document}
	\begin{abstract}
We present a conjecture for the Chow ring of the universal moduli stack of bundles over hyperelliptic curves and prove it for rank and genus two. 
Consequently, we obtain explicit generators and relations to conclude that the Chow ring is tautological.
In addition, we compute the Chow rings of products of universal Jacobians over genus two curves.
	\end{abstract}
	\maketitle
	\vspace{-1em}
\section*{Introduction}\label{intro}
The cohomology of the moduli space of semistable vector bundles over a curve has been the subject of extensive study \cite{ kirwan,zagier,king-newstead, kirwan-earl}. 
We have explicit sets of generators for all coprime ranks and degrees \cite{atiyah-bott}, relations for the cohomology ring \cite{kirwan-earl}, and formulas to compute Betti numbers \cite{del-bano}.
These results were later generalized to the moduli stack of vector bundles (including unstable) over a curve $C$, namely $Bun_{r,d}C$, in \cite{heinloth}.

Similar computations for universal moduli spaces of semistable bundles over curves were hypothesized to be accessible in the future,
 when the moduli spaces $\overline{\mathcal{M}_{g,n}}(X,\beta)$ of stable maps are better-understood and localisation techniques are used with greater effect, 
 by Kirwan in her ICM talk \cite[\S 8]{kirwan-talk}. 
 
However, much remains unkown about the universal moduli spaces to date due to the technical difficulties outlined in \cite{kirwan-talk}.
Recent computations of the Chow ring of the universal moduli space \cite{larson-18,larson-24,shubham-r2} have initiated progress in understanding their intersection theory.
Computations of the full Chow ring, in terms of explicit generators and relations, of the universal moduli stack of bundles over curves include bundles over genus zero curves \cite{larson-18} and universal Picard stacks over hyperelliptic curves \cite{larson-24}. 

The twisted-kappa classes, defined in \cite{larson-24, shubham-r2}, are tautological classes on the universal moduli stack of bundles. We show that they generate the Chow ring of the universal moduli stack of bundles of rank two over genus two curves and compute relations.
Our approach stands in contrast to the approach used by Heinloth \cite[\S 5]{heinloth} where the Poincare polynomial for the moduli space of semistable bundles over a curve was computed using the moduli stack.
We hope to extend this approach to higher ranks and genera through the following conjecture.

\begin{conj}{A}\label{conj:main}
	Let $\ve^\h({r,d,g})$ be the universal moduli stack of vector bundles of rank $r$, degree $d$ over hyperelliptic curves of genus $g$.
For any $g\geq 2, d\in\Z$ and a hyperelliptic curve $C$ of genus $g$, the composition of restriction and cycle-class map 
	$$A^*(\ve^\h({r,d,g}))\ri H_{a}^*(Bun_{r,d}C)$$ is an isomorphism onto $H^*_{\ka}$ where $H^*_{\ka}$ is the cohomology subring generated by twisted-kappa classes and $H^*_a$ is the algebraic cohomology subring.
\end{conj}

We prove the conjecture for $\ve(2,d,2)$. 
Analogous versions of the conjecture have further been verified for universal moduli spaces of slope-stable and Higgs stable bundles of rank two in \cite{shubham-higgs}.
For universal Jacobians over hyperelliptic curves, the conjecture is an immediate consequence of \cite[Theorem 1.1]{larson-24}.

The conjecture is closely related to the \textit{Franchetta Property}, which states that cycles in a family $\chi \ri B$ are homologically trivial when restricted to a very general fiber over $b\in B$ if and only if they are trivial in Chow of the same fiber. 
Effectively, the conjecture claims that the morphism $\ve^\h(r,d,g)\ri \h_g$ satisfies the Franchetta Property, 
and that the Gysin restriction map $A^*(\ve^\h(r,d,g))\ri A^*(Bun_{r,d}C)$ is injective for all $[C]\in |\h_g|$.
For more details, see \cite{laterveer-franchetta}. 

We apply excision on $\ve(2,d,2)$ given by the Harder-Narasimhan(HN) stratification \cite[Proposition 5.3]{heinloth}. 
The unstable HN-strata have been extensively studied in the form of filtration stacks \cite{prada}.
We compute the Chow rings of products of universal Jacobians over genus two curves and relate them to cycles on filtration stacks using \cite[Corollary 3.2]{prada}. The relevant notations are defined in \S\ref{notation}.
\begin{customthm}{A}
  Let $\pi_1,\pi_2$ be the projections onto the first and second factors of $J_*(0)\times_{M_*}J_*(0)$, and let $\mu:J_*(0)\times_{M_*}J_*(0)\ri J_*(0)$ be the twisting map. We have
  \begin{gather*}Im(A^*(J^0_2\times_{M_2}J^0_2)\ri A^*(J_*(0)\times_{M_*}J_*(0)))\simeq \Q[\pi_1^*\T_0,\pi_2^*\T_0, \mu^*\T_0]/
    (\pi_1^*\T_0^3,\pi_2^*\T_0^3, \mu^*\T_0^3,(\mu^*\T_0-\pi_2^*\T_0)\pi_1^*\T_0^2,\\(\mu^*\T_0-\pi_1^*\T_0)\pi_2^*\T_0^2, 
     \mu^*\T_0^2(\pi_1^*\T_0-\pi_2^*\T_0),\mu^*\T_0^2\pi_1^*\T_0+\pi_1^*\T_0\pi_2^*\T_0(\pi_1^*\T_0+\pi_2^*\T_0)-2\mu^*\T_0\pi_1^*\T_0\pi_2^*\T_0
	).\end{gather*}
\end{customthm}
We have the following theorems as a consequence to the conjecture.
\begin{customthm}{B}
  We have the following isomorphism for all $d\in \Z$:
  \begin{gather*}A^*(\ve({2,d,2}))\simeq \Q[\ka_{0,1,0},\ka_{0,0,1},\ka_{-1,0,1},\ka_{-1,2,0},\ka_{-1,0,2},\ka_{-1,1,1}]/(x_1^3,x_1^2x_3,x_1^2x_2+x_1x_3^2,6x_1x_2x_3+x_3^3,x_1x_3^3,\\
	x_1x_2^2+x_2x_3^2,x_2^2x_3,x_2x_3^3,x_2^3,x_3^5)\end{gather*}
  where $x_1 := \dfrac{d\ka_{0,1,0}-\ka_{-1,2,0}}{2},x_2 := \dfrac{\ka_{0,0,1}\ka_{-1,0,1}-\ka_{-1,0,2}}{2},x_3 := \dfrac{\ka_{0,1,0}\ka_{-1,0,1}+d\ka_{0,0,1}-2\ka_{-1,1,1}}{2}$.
\end{customthm}
\begin{customthm}{C}
	Given any $d\in \Z$, we have the following relations between higher twisted-kappa classes in $A^*(\ve({2,d,2}))$:
  \begin{gather*}
    2^{n-1}\ka_{-1,0,n} =\ka_{0,0,1}^{n-2}(n(n-1)\ka_{-1,0,2}-(n^2-2n)\ka_{-1,0,1}\ka_{0,0,1}) \forall n\geq 2,\\
    \ka_{-1,m,n} = \dfrac{1}{2^n}\ka_{0,0,1}^n\ka_{-1,m,0}+\dfrac{1}{2^m}\ka_{0,1,0}^m\ka_{-1,0,n}+\dfrac{nm}{2^{n+m-2}}\ka_{0,1,0}^{n-1}\ka_{0,0,1}^{m-1}(\ka_{-1,1,1}-\dfrac{1}{2}(\ka_{-1,0,1}\ka_{0,1,0}+d\ka_{0,0,1}))\forall m,n\geq 1,\\
   2^{m+n-1} \ka_{0,m,n} = \ka_{0,1,0}^m\ka_{0,0,1}^n\forall m,n\geq 0,\quad 
    \ka_{i,m,n} = 0\forall i\geq 2.
  \end{gather*}
\end{customthm}
	\subsection{Notation}\label{notation}
	\noindent All schemes in this paper are defined over the field of complex numbers. 
	All Chow rings have rational coefficients. We use the subspace convention for projective bundles. 
	
	We denote the locus of hyperelliptic curves with a marked Weierstrass point as $\mathcal{H}_{g,w}\hr \mathcal{M}_{g,1}$ for genus $g$ curves, and let $\mC_{g,w}$ be the universal curve over $\mathcal{H}_{g,w}$. 
	When $g=2$, we shall simply denote $\h_{2,w}$ as $\m_{2,w}$.
	For a fibre product $\tensor[_{i\in I}]{\prod}{_S}X_i$ over a scheme $S$, and a subset $J\subset I$, we denote the projection onto the product $\tensor[_{j\in J}]{\prod}{_S}X_j$ by $p_J$. 

	We adapt the following notation from \cite{shubham-r2} and \cite{larson-24}. The universal Picard stack of genus $g$ curves is denoted by $\J^d_g$, and the associated good moduli space is denoted by $J^d_g$.
	The moduli space of genus two curves with six marked Weierstrass points is denoted by $M_*$. There exists a universal family $\C_*\ri M_*$, the associated universal Jacobian for degree $d$ line bundles of which is denoted by $J_*(d)$.
	We denote the universal theta divisor on $J_*(1)$ by $\Theta_1$ and represent its image in $J_*(d)$ by $\T_d$ via the twisting morphisms defined using the first marked point.
	For a family of rank two bundles on curves $\E\ri\C_S\xrightarrow{\pi} S$, we denote twisted-kappa classes by $\ka_{a_0,a_1,a_2}:= \pi_*(K_\pi^{a_0+1}c_1(\E)^{a_1}c_2(\E)^{a_2})$.
\subsection*{Acknowledgements} 
I would like to thank my advisor Elham Izadi for her continued support and guidance. I am grateful to Jochen Heinloth, Soumik Ghosh, Qizheng Yin, and Shend Zhjeqi for helpful discussions.

\section{Computations on the universal moduli space}
The Chow ring of the universal moduli space of stable bundles with rank two, degree three over genus two curves, denoted by $U(2,3,2)$, was computed in \cite[Theorem B]{shubham-r2}. 
Using this, we compute the Chow ring of the universal moduli stack $\U(2,3,2)$.
\begin{theorem}\label{sec:gerbe-chow}
	The Chow ring $A^*(\U(2,3,2))$ is given by
	\begin{gather*}A^*(\U(2,3,2))=\Q[\ka_{0,1,0},\ka_{-1,2,0},\ka_{-1,0,1},\ka_{-1,1,1}]/(\T_U^3,H_U^2\T_U^2-4B_*\T_U^2,B_*H_U^2-B_*H_U\T_U+\dfrac{1}{2}B_*\T_U^2,\\
B_*H_U^2-B_*^2,H_U^3+H_U^2\T_U+\dfrac{1}{2}H_U\T_U^2-4B_*H_U)\text{ where }\\
H_U= \dfrac{1}{2}\left(4\ka_{-1,0,1}-\ka_{-1,2,0}\right), \T_U= \dfrac{1}{2}\left(3\ka_{0,1,0}-\ka_{-1,2,0}\right),\ka_{0,0,1} = \dfrac{1}{4}\ka_{0,2,0}+\dfrac{1}{8}(4\ka_{-1,0,1}-\ka_{-1,2,0})^2, \ka_{0,2,0} =\dfrac{1}{2}\ka_{0,1,0}^2, \\
\ka_{-1,3,0}= \dfrac{1}{4}\ka_{0,1,0}(6\ka_{-1,2,0}-9\ka_{0,1,0}), B_*= \dfrac{1}{2}\left(\dfrac{1}{2}(\kappa_{-1,2,0}-\kappa_{0,1,0})-\kappa_{-1,0,1}\right)^2- \dfrac{1}{2}\left( \dfrac{\kappa_{-1,3,0}}{3}-{\ka_{-1,1,1}}-\dfrac{\ka_{0,2,0}}{2}+{\ka_{0,0,1}}\right).\end{gather*}
\end{theorem}
\begin{proof}
We have $A^*(\mathcal{U}(2,3,2))\simeq A^*(U(2,3,2))$ since $\mathcal{U}(2,3,2)$ is a Deligne-Mumford stack.
Applying an argument similar to \cite[Lemma 6.5]{larson-24} and using generators of the Picard group from \cite[Theorem A.1]{fringuelli}, we have
$$A^*(\U(2,3,2))\simeq A^*(U(2,3,2))[\ka_{0,1,0}]$$
due to the $\mathbb{G}_m$-rigidification map $\nu_{2,3}: \U(2,3,2)\ri \mathcal{U}(2,3,2)$.
Hence, we get the claimed generators and relations. The claimed expressions for $\ka_{0,2,0},\ka_{-1,3,0}$ follow from \cite[Corollary 1.2]{larson-24}, and $\ka_{0,0,1}$ follows from \cite[Corollary 9]{shubham-r2}.
\end{proof}
\section{Product of universal Jacobians}\label{sec:jxj}
Following the notation in \cite{larson-24,shubham-r2}, let $J^d_2$ be the universal Jacobian of degree $d$ line bundles over genus $2$ curves. 
We consider the space $M_*\simeq M_{0,6}$, the induced map $M_*\ri M_2$ is defined using a family of curves $\C_*\ri M_*$ with sections $\{\sigma_i\}_{1\leq i\leq 6}$ defining the Weierstrass points of the fibers.
The relative Jacobian of degree $d$ for the family $\C_*\ri M_*$ is denoted by $J_*(d)$.\\
Let $J^d_{2,w}$ be the universal Jacobian of degree $d$ for the family of genus $2$ curves with a marked Weierstrass point.
\begin{definition}
	For any $n\in \mathbb{N}$, the fiber product $\times_{M_{2,w}}^nC_{2,w}$ has $i,j$-diagonals $\Delta_{ij}$ for all $1\leq i<j\leq n$, and $i,j$-conjugate diagonals $\tilde{\Delta}_{ij}$ for all $1\leq i\leq j\leq n$ defined by the equality of $i,j$-coordinates, and conjugacy of $i,j$-coordinates respectively. 
\end{definition}
\begin{lemma}\label{lemw:cxj}
	We have 
	$$A^*(\times_{M_{2,w}}^3C_{2,w}\setminus (\cup_{1\leq i<j\leq 3}\Delta_{ij}\cup_{1\leq i\leq j\leq 3}\tilde{\Delta}_{ij})) = \Q.$$
\end{lemma}
\begin{proof}
	We modify the proof in \cite[Lemma 4.8]{shubham-r2}. We consider the parametrisation of genus two curves with three marked points in \cite[\S 2.4]{casnati}.
A genus two curve with three distinct marked points $(C,p_1,p_2,p_3)$, with none of the $p_i$’s Weierstrass or conjugates under the hyperelliptic involution can be embedded as a plane quartic with a single nodal/cuspidal singularity given by the linear system $|2p_1+p_2+p_3|$.\\
We rigidify the embedding further by setting the images of the node $N\mapsto  [1 : 0 : 0],p_1\mapsto [0 : 1 : 0],p_2\mapsto [0 : 0 : 1],p_3 \mapsto [0 : 1 : 1]$. 
The equation of the image is then an element of the space
$$V := \{\la_1x_1^2x_2^2+\la_2x_1^2x_2x_3+\la_3x_1^2x_3^2+x_1x_2^3+\la_4x_1x_2^2x_3+\la_5x_1x_2x_3^2+\la_6x_1x_3^3+ x_2(x_2-x_3)x_3^2\}.$$
 We find images of Weierstrass points under this map. Weierstrass points correspond to lines $Z(x_2-cx_3)$ for $c\in \K$ tangent to the quartic. Therefore, $c$ must be a root of the polynomial 
 \begin{equation}\label{c-3}(c^3+c^2\la_4+c\la_5+\la_6)^2 = 4c(c-1)(c^2\la_1+c\la_2+\la_3)\tag{$E_{(C,p_1,p_2,p_3)}$}\end{equation} 
Let $U\subset \mathbb{A}_\K^6$ be the open subset of points parametrising ordered tuples $(\la_i)_{1\leq i\leq 6}$ corresponding to equations of nodal/cuspidal plane quartics.
Given any $\{\la_i\}_{2\leq i\leq 6}\in \K,c\in \K\setminus\{0,1\}$, we can uniquely determine $\la_1$ so that an equation in the form \ref{c-3} is satisfied.\\
Let $U'\subset \mathbb{A}_\K^6$ be the corresponding open subset parametrising ordered tuples $(c,\la_2,\cdots,\la_6)$ such that the associated tuple $(\la_i)_{1\leq i\leq 6}$ is contained in $ U$.
Using the blowup construction in \cite[Lemma 3.4]{shubham-r2}, we have a family of genus $2$ curves defined over $U'$, with a Weierstrass section defined using $c$.\\
Consequently, we have a finite, surjective map $U'\ri \times_{M_{2,w}}^3C_{2,w}\setminus (\cup_{1\leq i<j\leq n}\Delta_{ij}\cup_{1\leq i\leq j\leq n}\tilde{\Delta}_{ij})$. Hence, we have the claimed Chow triviality by \cite[Remark 2.2]{vakil}.
\end{proof}
\begin{corollary}\label{cor:cj}
	The pullback map 
	$$A^*(C_2\times_{M_2}J^d_2)\ri A^*(C_{2,w}\times_{M_{2,w}}J^d_{2,w})$$
	is an isomorphism for all values of $d\in \Z$.
\end{corollary}
\begin{proof}
	The pullback map is injective due to the finiteness, and surjectivity of the morphism
	$C_{2,w}\times_{M_{2,w}}J^d_{2,w}\ri C_2\times_{M_2}J^d_2$. 
	For $d=3$, the surjectivity of the pullback map follows from Lemma \ref{lemw:cxj}, and \cite[Theorem 4.4]{shubham-r2}.\\
	Using the Weierstrass section of $\mathcal{C}_{2,w}\ri\mathcal{M}_{2,w}$, we have isomorphisms $C_{2,w}\times_{M_{2,w}}J^d_{2,w}\simeq C_{2,w}\times_{M_{2,w}}J^3_{2,w}$ for all $d\in\Z$. 
	Hence, we have the claimed isomorphisms of pullback maps for all values of $d\in\Z$.
\end{proof}
\begin{lemma}\label{lem:c4}
	We have 
	$$A^*(\times_{M_{2,w}}^4C_{2,w}\setminus (\cup_{1\leq i<j\leq 4}\Delta_{ij}\cup_{1\leq i\leq j\leq 4}\tilde{\Delta}_{ij})) = \Q.$$
\end{lemma}
\begin{proof}
A genus two curve with four distinct marked points $(C,p_1,p_2,p_3,p_4)$, with none of the $p_i$’s Weierstrass or conjugates under the hyperelliptic involution can be embedded as a plane quartic with a single nodal/cuspidal singularity given by the linear system $|p_1+p_2+p_3+p_4|$.\\
We rigidify the embedding further by setting the images of the node $N\mapsto  [1 : 0 : 0],p_1\mapsto [0 : 1 : 0],p_2\mapsto [0 : 0 : 1],p_3 \mapsto [0 : 1 : 1]$. 
The equation of the image is then an element of the space
$$V := \{\la_1x_1^2x_2^2+\la_2x_1^2x_2x_3+\la_3x_1^2x_3^2+x_1x_2^3+\la_4x_1x_2^2x_3+\la_5x_1x_2x_3^2+\la_6x_1x_3^3+ x_2x_3(x_2-x_3)(x_2-\la_7x_3)\}$$
 We find images of Weierstrass points under this map. Weierstrass points correspond to lines $Z(x_2-cx_3)$ for $c\in \K$ tangent to the quartic. Therefore, $c$ must be a root of the polynomial 
 \begin{equation}\label{c-4}(c^3+c^2\la_4+c\la_5+\la_6)^2 = 4c(c-1)(c-\la_7)(c^2\la_1+c\la_2+\la_3)\tag{$E_{(C,p_1,p_2,p_3,p_4)}$}\end{equation} 
Let $U\subset \mathbb{A}_\K^7$ be the open subset of points parametrising ordered tuples $(\la_i)_{1\leq i\leq 7}$ corresponding to equations of nodal/cuspidal plane quartics.
Given any $\{\la_i\}_{2\leq i\leq 7}\in \K,c\in \K\setminus\{0,1,\la_7\}$, we can uniquely determine $\la_1$ so that an equation in the form \ref{c-4} is satisfied.\\
Let $U'\subset \mathbb{A}_\K^7$ be the corresponding open subset parametrising ordered tuples $(c,\la_2,\cdots,\la_7)$ such that the associated tuple $(\la_i)_{1\leq i\leq 6}$ is contained in $ U$.
Using the blowup construction in \cite[Lemma 3.4]{shubham-r2}, we have a family of genus two curves defined over $U'$, with a section of Weierstrass point defined using $c$.\\
Consequently, we have a finite, surjective map $U'\ri \times_{M_{2,w}}^4C_{2,w}\setminus (\cup_{1\leq i<j\leq n}\Delta_{ij}\cup_{1\leq i\leq j\leq n}\tilde{\Delta}_{ij})$. Hence, we have the claimed Chow triviality by \cite[Remark 2.2]{vakil}.
\end{proof}
\begin{definition}
	Let $S_{0}^*$ be the image of the pullback $ A^*(J^0_{2,w}\times_{M_2}J^0_{2,w})\ri A^*(J_*(0)\times_{M_*}J_*(0))$.
\end{definition}
\begin{theorem}\label{jxj}
  Let $\pi_1,\pi_2$ be the projections onto first and second factors of $J_*(0)\times_{M_*}J_*(0)$, and let $\mu:J_*(0)\times_{M_*}J_*(0)\ri J_*(0)$ be the twisting map. We have
  \begin{gather*}S_{0}^*\simeq \Q[\pi_1^*\T_0,\pi_2^*\T_0, \mu^*\T_0]/
    (\pi_1^*\T_0^3,\pi_2^*\T_0^3, \mu^*\T_0^3,(\mu^*\T_0-\pi_2^*\T_0)\pi_1^*\T_0^2,(\mu^*\T_0-\pi_1^*\T_0)\pi_2^*\T_0^2, \\
     \mu^*\T_0^2(\pi_1^*\T_0-\pi_2^*\T_0),\mu^*\T_0^2\pi_1^*\T_0+\pi_1^*\T_0\pi_2^*\T_0(\pi_1^*\T_0+\pi_2^*\T_0)-2\mu^*\T_0\pi_1^*\T_0\pi_2^*\T_0
	).\end{gather*}
\end{theorem}
\begin{proof}
	By \cite[Theorem 3.3]{shubham-r2}, we have that $\pi_1^*\T_0,\pi_2^*\T_0,\mu^*\T_0\in S_{0}^*$.
The pullback map $A^*(J^0_2\times_{M_2}J^{0}_2)\hr A^*(J_*(0)\times_{M_*}J_*(0))$ is injective by \cite[Remark 2.2]{vakil}.
We show that the image of the pullback is generated by $\pi_1^*\T_0,\pi_2^*\T_0,\mu^*\T_0$, and that the claimed ideal of relations generates the ideal of relations in $S^*_0$.
We have the canonical map 
$$\times_{M_{2,w}}^4C_{2,w}\xrightarrow{\Phi} J^2_{2,w}\times_{M_{2,w}}J^2_{2,w}\simeq J^0_{2,w}\times_{M_{2,w}}J^0_{2,w}.$$
given by the sum of $(1,2)$ and $(3,4)$ coordinates.\\
From the definition in $\Phi$, and \cite[Theorem 3.3]{shubham-r2}, it follows that the image $\Phi_*(A_*(\tilde{\Delta}_{12}))$ is linearly spanned by $\{\pi_1^*\T_0^2,\pi_1^*\T_0^2\pi_2^*\T_0,\pi_1^*\T_0^2\pi_2^*\T_0^2\}$.
Similarly, the image $\Phi_*(A_*(\tilde{\Delta}_{34}))$ is linearly spanned by $\{\pi_2^*\T_0^2,\pi_2^*\T_0^2\pi_1^*\T_0,\pi_2^*\T_0^2\pi_1^*\T_0^2\}$. An inductive argument using \cite[Lemma 4.7 and Theorem 4.4]{shubham-r2} shows that $S_0^*$ is generated by the cycles $\{\Phi_*([\Delta_{ij}]),\Phi_*([\tilde{\Delta}_{ij}])\}_{1\leq i\leq j\leq 4}$ and their transverse intersections. Therefore, its enough to show that each of these classes are in the claimed subring.\\
Let $d: J_*(0)\times_{M_*}J_*(0)\ri J_*(0)$ be the difference map given by $(\lb_1,\lb_2)\mapsto \lb_1\otimes\lb_2^\vee$.\\
Restricting to a fiber of $\pi_2$, and applying the theorem of the square, we have that there exists a rational $k$ such that 
\begin{equation}\label{eq:mu,d}
2\pi_1^*\T_0+k\pi_2^*\T_0 = \mu^*\T_0+d^*\T_0\implies k=2\implies 2(\pi_1^*\T_0+\pi_2^*\T_0) = \mu^*\T_0+d^*\T_0
\end{equation}
since $\mu^*\T_0+d^*\T_0$ is invariant under the transposition involution on $J_*(0)\times_{M_*}J_*(0)$.
Similarly, we have 
\begin{gather*}\Phi_*([\Delta_{ij}]) = d^*\T_0+\pi_1^*\T_0+\pi_2^*\T_0,
\Phi_*([\tilde{\Delta}_{ij}])= \mu^*\T_0+\pi_1^*\T_0+\pi_2^*\T_0\forall i\in \{1,2\},j\in\{3,4\}.\end{gather*}
The following equations show that the claimed subring is equal to $S_0^*$:
$$\Phi_*([\Delta_{12}]) = 4\pi_1^*\T_0,\Phi_*([\Delta_{34}]) = 4\pi_2^*\T_0,\Phi_*([\tilde{\Delta}_{ii}]) = \begin{cases}3\pi_1^*\T_0 \text{ if }i\in\{1,2\}\\3\pi_2^*\T_0\text{ if }i\in \{3,4\}\end{cases}.$$
Let $\Delta_J:= J_*(0)\hr J_*(0)\times_{M_*}J_*(0)$ be the diagonal. Using (\ref{eq:mu,d}), we have 
\begin{equation}\label{eq:delta-j}[\Delta_J] = \dfrac{1}{2}d^*\T_0^2 = \dfrac{1}{2}(2(\pi_1^*\T_0+\pi_2^*\T_0)-\mu^*\T_0)^2.\end{equation}
Therefore, $S_0^*$ is generated by $\pi_1^*\T_0,\pi_2^*\T_0,\mu^*\T_0$.
We verify the claimed relations. It is clear that $$\pi_1^*\T_0^3 = \pi_2^*\T_0^3=\mu^*\T_0^3 = 0.$$
Furthermore, we have $$\mu^*\T_0\pi_1^*\T_0^2 = \pi_{2}^*\T_0\pi_1^*\T_0^2 = 2[\{(\mo,\mo(x-\s_1))|x\in \C_*\}],\mu^*\T_0^2\pi_1^*\T_0 = \mu^*\T_0^2\pi_2^*\T_0 = 2[\{(x-\s_1,\s_1-x)|x\in\C_*\}].$$
The last claimed relation follows from the observation that $\mu^*\T_0\pi_1^*\T_0\pi_2^*\T_0$ has three integral components:
$$\underbrace{[\{(\mo(x-\s_1,\s_1-x))|x\in \C_*\}]}_{\dfrac{1}{2}\mu^*\T_0^2\pi_1^*\T_0},\underbrace{[\{(\mo,\mo(y-\s_1))|y\in\C_*\}]}_{\dfrac{1}{2}\pi_1^*\T_0^2\pi_2^*\T_0},\underbrace{[\{(\mo(z-\s_1),\mo)|z\in\C_*\}]}_{\dfrac{1}{2}\pi_2^*\T_0^2\pi_1^*\T_0}.$$
To see that the claimed relations generate the ideal, we observe that the claimed quotient has the following linear basis, \text{where we identify }$(\pi_1,\pi_1^*\T_0), (\pi_2,\pi_2^*\T_0),(\mu,\mu^*\T_0)$:
\begin{gather*}\{1,\underbrace{\pi_1,\pi_2,\mu}_{A^1}, \underbrace{\pi_1^2,\pi_2^2,\mu^2,\pi_1\pi_2,\pi_1\mu,\pi_2\mu}_{A^2},\underbrace{\mu^2\pi_1, \pi_1\pi_2^2,\pi_1^2\pi_2}_{A^3},\underbrace{\mu^2\pi_1\pi_2}_{A^4}\}.\end{gather*}
The claim now follows from the observation that the cycles above are linearly independent in $H^*(J^0_C\times J^0_C)$ for a genus two curve $C$.
\end{proof}
\begin{corollary}
	For all $d,d'\in\Z$, we have $$A^*(J^d_2\times_{M_2}J^{d'}_2)\simeq A^*(J^0_2\times_{M_2}J^0_2).$$
\end{corollary}
\begin{proof}
	From Theorem \ref{jxj}, we have that the pullback $A^*(J^d_2\times_{M_2}J^{d'}_2)\ri A^*(J^d_{2,w}\times_{M_{2,w}}J^{d'}_{2,w})$ is an isomorphism
due to the canonical isomorphism $J^d_2\times_{M_2}J^{d'}_2\simeq J^0_{2,w}\times_{M_{2,w}}J^0_{2,w}$. The claim now follows from the observation that $S_0^*$ coincides with the image of the pullback $A^*(J^0_2\times_{M_2}J^{0}_2)\ri A^*(J_*(0)\times_{M_*}J_*(0))$.
\end{proof}
\begin{rem}
	A motivic argument using \cite[Proposition 0.1]{bae-maulik} produces a more general proof for the above corollary for arbitrary genus.
\end{rem}
\begin{corollary}\label{cor:gerbe-jxj}
For all $d,d'\in \mathbb{Z}$, there exist $T_1,T_2\in A^1(\J^d_2\times_{\mathcal{M}_2}\J^{d'}_2)$ such that
$$A^*(\J^d_2\times_{\mathcal{M}_2}\J^{d'}_2)\simeq A^*(\jac^d_2\times_{M_2}\jac^{d'}_2)[T_1,T_2].$$
\end{corollary}
\begin{proof}
It is well-known that $\J^d_2\ri\jac^d_2$ is a $\mathbb{G}_m$-gerbe. 
Consequently, $\J^d_2\times_{\mathcal{M}_2}\J^{d'}_2\ri \jac^d_2\times_{M_2}\jac^{d'}_2$ is a $\mathbb{G}_m\times \mathbb{G}_m$-gerbe.
Applying \cite[Lemma 3.1.1.8]{lieblich}, we have
$$A^*(\J^d_2\times_{\mathcal{M}_2}\J^{d'}_2)\simeq A^*(\jac^d_2\times_{M_2}\jac^{d'}_2)\otimes A^*(B\mathbb{G}_m\times B\mathbb{G}_m)\simeq A^*(J^d_2\times_{M_2}J^{d'}_2)\otimes A^*(B\mathbb{G}_m\times B\mathbb{G}_m).$$
Using \cite[Example 4.3]{heinloth}, there exist $T_1,T_2\in A^1(B\mathbb{G}_m\times B\mathbb{G}_m)$ such that
$$A^*(B\mathbb{G}_m\times B\mathbb{G}_m)\simeq \Q[T_1,T_2].$$
\end{proof}
\section{The filtration stack and the stratification}
To understand $A^*(\ve(r,d,g))$, it is crucial to construct stacks parameterizing extensions to apply results from \S\ref{sec:jxj}.
For a \textit{Harder-Narasimhan polygon}(HN-polygon) of type $t = (n_i,d_i)_i$, the stack of vector bundles with filtration of type $t$ is denoted by $\ve(r,d,g)^t$.
The stack $\ext^{\underline{d}}_{\underline{n}}(g)$ has been defined in \cite[Lecture 5]{heinloth} as the stack of filtered vector bundles of type ${t} = (\underline{n},\underline{d})$.
More precisely:
\begin{enumerate}
	\item For a scheme $S$, the objects are given by $\ext^{\underline{d}}_{\underline{n}}(g)(S):= (\C\ri S, \{\e_1\subset \e_2\subset\cdots\}| rk(\e_i) = n_i, deg(\e_i) = d_i)$ where $\C\ri S$ is a family of smooth curves of genus $g$, $\e_i$ are vector bundles on $C$. 
	\item This is an algebraic stack such that the forgetful morphism $\ext^{\underline{d}}_{\underline{n}}(g)\ri \ve(r,d,g)$ is representable.
\item Quotients of the filtration define the forgetful morphism $\ext^{\underline{d}}_{\underline{n}}(g)\ri \times_{M_g}^m\ve(n_i,d_i,g)$, where $m$ is the length of the filtration.
\end{enumerate}
For $\ve(2,d,g)$, the unstable HN-polygons are of the form $t = ((d_1,d),(1,2))$ for $d_1>d/2$. 
Consequently, the filtration stacks are isomorphic to the corresponding extension stacks of universal Picards. We have the morphisms
$$\theta_{d_1,d}:\ext^{(d_1,d)}_{(1,2)}(g)\ri\J^{d_1}_g\times_{\mathcal{M}_g}\J^{d-d_1}_g$$ 
such that $\theta_{{d_1},d}$ is a vector bundle stack of relative dimension $g-1+d-2d_1$ for all $d_1,d\in \Z$ \cite[Corollary 3.2]{prada}.

For genus two and degree three, we have   
$$\ext^{(d,3)}_{(1,2)}(2)\ri \ve(2,d,2)$$ is an immersion \cite[Proposition 5.3]{heinloth}.
The HN-stratification of $\ve(2,3,2)$ then gives the following lemma. 
\begin{lemma}\label{lem:strata}
We have the following stratification of $\ve(2,3,2)$:
$$\ve(2,3,2) = \U(2,3,2) \sqcup_{n\geq 2} \ext^{(n,3)}_{(1,2)}(2).$$
\end{lemma}
\begin{rem}\label{rem:strata}
For all $n\geq 2$, we have $$A^*(\ext^{(n,3)}_{(1,2)}(2))\simeq A^*(\J^n_2\times_{\mathcal{M}_2}\J^{3-n}_2)$$
by \cite[Lemma 3.3]{prada}. The relative dimension of $\theta_{n,3}$ further implies that $\ext^{(n,3)}_{(1,2)}$ has codimension $2n-2$ in $\ve(2,3,2)$. 
\end{rem}
\subsection{Hilbert series of defined strata}\label{sec:hilb-series}
\begin{definition}
The Hilbert series of an algebraic stack $\chi$ is the Hilbert series of graded group 
$A_*(\chi)$:
$$hilb_{ser}(\chi):=\Sigma_{i\geq 0} T^i\dim A_{\dim \chi - i}(\chi).$$ With slight abuse of notation, we shall denote the Hilbert series of a graded group $G$ by $hilb_{ser}(G)$.
\end{definition}
\begin{theorem}\label{thm:hilb-strata}
	For all $n\geq 2$, the Hilbert series of the strata are given by
	$$hilb_{ser}(\ext^{(n,3)}_{(1,2)}(2)) = \dfrac{1+3T+6T^2+3T^3+T^4}{(1-T)^2}$$
\end{theorem}
\begin{proof}
Follows from Corollary \ref{cor:gerbe-jxj} and Theorem \ref{jxj}.	
\end{proof}
\section{The conjecture in genus two}
We prove Conjecture \ref{conj:main} for odd degree bundles and use a Hecke correspondence to show it for even degree.
\subsection{Odd degree}
We shall prove the conjecture for degree three bundles. For simplicity, we shall set $d$ to be $3$ as explained in \cite[\S 4]{shubham-r2}.
\begin{definition}
For two series with integral coefficients $\alpha(T) = \Sigma_{j\geq 0}\alpha_jT^j,\beta(T) = \Sigma_{k\geq 0}\beta_kT^k$, we shall denote $\alpha(T)\preceq\beta(T)$ if $\alpha_i\leq \beta_i$ for all $i\geq 0$.
\end{definition}
\begin{theorem}\label{thm:main}
		For a curve $C$ of genus $2$, the following map is an isomorphism onto $H^*_{\ka}$ 
	$$A^*(\ve({2,3,2}))\ri H_{a}^*(Bun_{2,3}C)$$
	where $H^*_{\ka}$ is the cohomology subring generated by twisted-kappa classes and $H^*_a$ is the algebraic cohomology subring. 
\end{theorem}
\begin{proof}
	We start by observing that 
$$A^*(\ve(2,3,2))\supset R^*(\ve(2,3,2))\twoheadrightarrow H^*_\ka(Bun_{2,3}C).$$
Therefore, we have $$hilb_{ser} (H^*_\ka(Bun_{2,3}C))\preceq hilb_{ser}(A^*(\ve(2,3,2))).$$
We shall show that these Hilbert series are in fact equal. The theorem would then follow as an immediate consequence of the following:
\begin{itemize}
	\item $A^*(\ve(2,3,2)) = R^*(\ve(2,3,2))$.
	\item The surjection $R^*(\ve(2,3,2))\ri H^*_\ka(Bun_{2,3}C)$ is an isomorphism.
\end{itemize}
In Figure \ref{cohom:code}, we use generators of $H^*_\ka(Bun_{2,3}C)$ from \cite[Theorem 4.7]{heinloth}. We compute the Hilbert series of the associated quotient ring (denoted by \textit{coh}) in Figure \ref{conj:code}. Using Figure \ref{cohom:code}, a linear basis of the subring of $H^*_\ka(Bun_{2,3}C)$ generated by $\{p_1,p_2,p_3\}$ is 
$$\{1,p_1,p_1^2, p_3, p_1p_3 ,p_2 ,p_1p_2 ,p_3^2 ,p_1p_3^2, p_2p_3,p_2^2, p_3^3,p_2p_3^2,p_3^4 \}$$ 
which matches with the numerator of \textit{coh}. Hence, we have that the claimed relations generate the ideal of relations in $H^*_\ka(Bun_{2,3}C)$. Therefore,
$$hilb_{ser}(H^*_\ka(Bun_{2,3}C)) = \dfrac{1+T+2T^2+2T^3+2T^4+2T^5+2T^6+T^7+T^8}{(1-T^2)(1-T)^2}.$$
On the other hand, we have the following using Lemma \ref{lem:strata}, Remark \ref{rem:strata} and excision on $\ve(2,3,2)$:
\begin{gather*}hilb_{ser}(\ve(2,3,2)) \preceq hilb_{ser}(\U(2,3,2))+\Sigma_{n\geq 2}T^{2n-2}hilb_{ser}(\ext_{(1,2)}^{(n,3)}(2)).\end{gather*}
Applying computations of Hilbert series from Theorem \ref{thm:hilb-strata}, we have the expression for \textit{r} in Figure \ref{conj:code}.
We denote $hilb_{ser}(\U(2,3,2))$ by \textit{s} and compute in Figure \ref{conj:code}. The equality between \textit{coh} and \textit{r}+\textit{s}, verified by Figure \ref{conj:code}, proves the theorem.
\end{proof}
\subsection{The Hecke correspondence}\label{sec:hecke}
We construct the well-known Hecke correspondence between $\ve_w^\h(2,3,2),\ve_w^\h(2,2,2)$ to prove Conjecture \ref{conj:main} for even degree bundles.
\begin{definition}\label{def:pu}
	Let $w:\mathcal{C}_{2,w}\ri \mathcal{M}_{2,w}$ be the universal Weierstrass section. For any integer $d$, let $\ve_w(2,d,2)$ be pullbacks of $\ve(2,d,2)$ over $\mathcal{M}_{2,w}\ri\mathcal{M}_2$, 
	let $\e_d$ be the Poincar\'e bundles on the stacks $\mathcal{C}_{2,w}\times_{\mathcal{H}_{2,w}} \ve_w(2,d,2)$, and set $\p_{d}:=\p_{\ve_w(2,d,2)}(\e_d|^\vee_{w\times_{\mathcal{H}_{2,w}} \ve_w^\h(2,d,2)})$. 
Additionally, we denote the associated map $\p_{d}\ri \ve_w(2,d,2)$ by $\phi_d$.
\end{definition}
\begin{theorem}\label{thm:cu}
The Hecke correspondence along the Weierstrass section induces an isomorphism $c_w:\p_{3}\xrightarrow{\sim} \p_{2}$.
\end{theorem}
\begin{proof}
	The Poincar\'e bundle $\e_3\ri \mathcal{C}_{2,w}\times_{\mathcal{M}_{2,w}} \ve_w^\h(2,3,2)$ pulls back via $(1\times \phi_3)$ to the following composition of surjections:
	$$(1\times \phi_3)^*\e_3\twoheadrightarrow \phi_3^*(\e_3|_{w\times{\mathcal{M}_{2,w}} \ve_w^\h(2,3,2)})\twoheadrightarrow \mo_{w\times{\mathcal{M}_{2,w}}\p_3}(1).$$
Let $\e_2$ be the kernel of the composed surjection $(1\times \phi_3)^*\e_3\twoheadrightarrow\mo_{w\times_{\mathcal{M}_{2,w}}\p_3}(1)$. We show that $\e_2\ri \mathcal{C}_{2,w}\times_{\mathcal{M}_{2,w}} \p_3$ is a family of vector bundles with the claimed properties.\\
The sheaves $(1\times \phi_3)^*\e_3,\mo_{w\times_{\mathcal{M}_{2,w}} \p_3}(1)$ are flat over $\p_3$. Therefore, 
$$\e_2|_{C\times p}\simeq ker(\e_3|_{C_p\times \phi_3(p)}\twoheadrightarrow \mo_{w_p}) \forall p\in |\p_3|.$$
It follows that $deg(\e_2|_{C_p\times \phi_3(p)})=2 $ for all $p\in \p_3$. Hence we a have a morphism $c_w:\p_3\ri \p_2$.
To see that $c_w$ is an isomorphism, we construct the inverse map. The Hecke modification of $\e_2$ along the section $w$ produces a family of degree $1$ bundles on $\mC_{2,w}\times_{\m_{2,w}}\p_2$, twisting by $\mo(w)$ defines the inverse map $\p_2\ri \p_3$.
\end{proof}
\begin{corollary}\label{cor:hecke}
We have $hilb_{ser}(\ve_w(2,3,2)) = hilb_{ser}(\ve_w(2,2,2))$.
\end{corollary}
\begin{proof}
Using the projective bundle formula, we have 
\begin{gather*}hilb_{ser}(\ve_w(2,3,2))(1+T) = hilb_{ser}(\p_3)=hilb_{ser}(\p_2) = (1+T)hilb_{ser}(\ve(2,2,2))\\
	\implies hilb_{ser}(\ve_w(2,3,2)) = hilb_{ser}(\ve_w(2,2,2)).\end{gather*}
\end{proof}
\begin{corollary}\label{cor:conj}
Conjecture \ref{conj:main} holds for $A^*(\ve(2,d,2))$ for $d$ even.
\end{corollary}
\begin{proof}
As shown in \cite[Theorem 4.7]{heinloth}, we have a non-canonical isomorphism
$$H^*_\ka(Bun_{2,3}C)\simeq H^*_\ka(Bun_{2,2}C).$$
Hence, we have that $hilb_{ser}(H^*_\ka(Bun_{2,2}C))=hilb_{ser}(\ve(2,2,2))$ by Corollary \ref{cor:hecke}. Using similar arguments as in Theorem \ref{thm:main}, we have the claimed isomorphism.
\end{proof}
\section{Relations and higher $\ka$-classes in terms of generators}
\begin{theorem}\label{thm:taut-high}
	For any integer $d$, we have the following relations between higher twisted-kappa classes in $A^*(\ve({2,d,2}))$:
  \begin{gather*}
    2^{n-1}\ka_{-1,0,n} =\ka_{0,0,1}^{n-2}(n(n-1)\ka_{-1,0,2}-(n^2-2n)\ka_{-1,0,1}\ka_{0,0,1}) \forall n\geq 2,\\
    \ka_{-1,m,n} = \dfrac{1}{2^n}\ka_{0,0,1}^n\ka_{-1,m,0}+\dfrac{1}{2^m}\ka_{0,1,0}^m\ka_{-1,0,n}+\dfrac{nm}{2^{n+m-2}}\ka_{0,1,0}^{n-1}\ka_{0,0,1}^{m-1}(\ka_{-1,1,1}-\dfrac{1}{2}(\ka_{-1,0,1}\ka_{0,1,0}+d\ka_{0,0,1}))\forall m,n\geq 1,\\
   2^{m+n-1} \ka_{0,m,n} = \ka_{0,1,0}^m\ka_{0,0,1}^n\forall m,n\geq 0,\quad 
    \ka_{i,m,n} = 0\forall i\geq 1.
  \end{gather*}
\end{theorem}
\begin{proof}
We begin by establishing the third relation. Pulling back to the finite cover $\ve_w(2,d,2)\ri \ve(2,d,2)$, we have
\begin{gather*}\dfrac{1}{2}\ka_{0,m,n} = p_{2*}i_{w*}(c_1(\e_d|_{w_\V})^mc_2(\e_d|_{w_\V})^n) = c_1(\e_d|_{w_\V})^mc_2(\e_d|_{w_\V})^n = \\
\left(\dfrac{1}{2}p_{2*}(K_{p_2}c_1(\e_d))\right)^m\left(\dfrac{1}{2}p_{2*}(K_{p_2}c_2(\e_d))\right)^n= \dfrac{1}{2^{m+n}}\ka_{0,1,0}^m\ka_{0,0,1}^n\forall m,n\geq 0.\end{gather*}
where $i_w: \ve_w(2,d,2)\hr \mC_{2,w}\times_{\mathcal{M}_{2,w}}\ve_w(2,d,2)$ is the pullback of the Weierstrass section, $p_2:\mC_{2,w}\times_{\mathcal{M}_{2,w}}\ve_w(2,d,2)\ri\ve_w(2,d,2)$ is the projection map, and $w_\V$ is the image of $i_w$.

We now prove the remaining relations using Conjecture \ref{conj:main} for $\ve(2,d,2)$. The $\ka$-classes denote the corresponding elements of $H^*(Bun_{2,d}C)$ from hereon.

Let $\e_{d,C}\ri C\times Bun_{r,d}C$ be a restriction of $\e_d$. Following the noration in \cite[Theorem 4.7]{heinloth}, for $1\leq i\leq 2$, we have
$$c_i(\e_{d,C}) = 1\otimes a_i+\Sigma_{j=1}^4\gamma_j\otimes b_i^j+[pt]\otimes f_i$$
with $(\gamma_j)_{1\leq j\leq 4}\subset H^1(C,\Q)$ a symplectic basis.
It is clear that $a_i,f_i\in H^*_\ka(Bun_{r,d}C)$. More precisely, we have 
$$a_1 = \dfrac{1}{2}\ka_{0,1,0}, a_2 = \dfrac{1}{2}\ka_{0,0,1}, f_1 = d, f_2 = \ka_{-1,0,1}.$$
On the other hand, we have 
\begin{gather*}
	b_1^1b_1^3+b_1^2b_1^4 = \dfrac{d\ka_{0,1,0}-\ka_{-1,2,0}}{2},
	b_2^1b_2^3+b_2^2b_2^4 = \dfrac{\ka_{0,0,1}\ka_{-1,0,1}-\ka_{-1,0,2}}{2} ,\\
	b_1^1b_2^3+b_1^2b_2^4-b_1^3b_2^1-b_1^4b_2^2 = \dfrac{\ka_{0,1,0}\ka_{-1,0,1}+d\ka_{0,0,1}-2\ka_{-1,1,1}}{2}.
\end{gather*}
It follows that 
\begin{gather*}\ka_{-1,0,n} = na_2^{n-1}f_2-2{n\choose 2}a_2^{n-2}(b_2^1b_2^3+b_2^2b_2^4) = \dfrac{n}{2^{n-1}}\ka_{0,0,1}^{n-1}\ka_{-1,0,1}-\dfrac{n(n-1)}{2^{n-1}}\ka_{0,0,1}^{n-2}(\ka_{0,0,1}\ka_{-1,0,1}-\ka_{-1,0,2})\\
\implies \ka_{-1,0,n} = \dfrac{\ka_{0,0,1}^{n-2}}{2^{n-1}}((n^2-n)\ka_{-1,0,2}-(n^2-2n)\ka_{0,0,1}\ka_{-1,0,1})\forall n\geq 2,\\
\ka_{-1,m,n} = a_1^m\ka_{-1,0,n}+a_2^n\ka_{-1,m,0} - mna_1^{m-1}a_2^{n-1}(b_1^1b_2^3+b_1^2b_2^4 -b_1^3b_2^1 - b_1^4b_2^2)\implies \\
 \ka_{-1,m,n}=\dfrac{1}{2^n}\ka_{0,0,1}^n\ka_{-1,m,0}+\dfrac{1}{2^m}\ka_{0,1,0}^m\ka_{-1,0,n}+\dfrac{nm}{2^{n+m-2}}\ka_{0,1,0}^{n-1}\ka_{0,0,1}^{m-1}(\ka_{-1,1,1}-\dfrac{1}{2}(\ka_{-1,0,1}\ka_{0,1,0}+d\ka_{0,0,1}))\forall m,n\geq 1.
\end{gather*}
\end{proof}
\begin{theorem}
  Conjecture \ref{conj:main} gives the following presentation for $A^*(\ve(2,d,2))$:
  \begin{gather*}A^*(\ve({2,d,2}))\simeq \Q[\ka_{0,1,0},\ka_{0,0,1},\ka_{-1,0,1},\ka_{-1,2,0},\ka_{-1,0,2},\ka_{-1,1,1}]/(x_1^3,x_1^2x_3,x_1^2x_2+x_1x_3^2,6x_1x_2x_3+x_3^3,x_1x_3^3,\\
	x_1x_2^2+x_2x_3^2,x_2^2x_3,x_2x_3^3,x_2^3,x_3^5)\end{gather*}
  where $x_1 := \dfrac{d\ka_{0,1,0}-\ka_{-1,2,0}}{2},x_2 := \dfrac{\ka_{0,0,1}\ka_{-1,0,1}-\ka_{-1,0,2}}{2},x_3 := \dfrac{\ka_{0,1,0}\ka_{-1,0,1}+d\ka_{0,0,1}-2\ka_{-1,1,1}}{2}$.
\end{theorem}
\begin{proof}
	Using Theorem \ref{thm:taut-high} and \cite[Corollary 1.2]{larson-24}, we have that $H^*_\ka(Bun_{2,d}C)$ is generated by 
$$\{a_1,a_2,f_2, b_1^1b_1^3+b_1^2b_1^4, b_2^1b_2^3+b_2^2b_2^4,b_1^1b_2^3+b_1^2b_2^4-b_1^3b_2^1-b_1^4b_2^2\}.$$
Using Theorem \ref{thm:main} and Figure \ref{cohom:code}, we have the claimed relations. The Hilbert series for the claimed presentation matches that of $A^*(\ve(2,d,2))$ as explained in Theorem \ref{thm:main}, establishing the isomorphism.
\end{proof}
\begin{rem}
	It is worth noting that the conjecture has been verified for $\J^d_{2,g}$ \cite[Theorem 1.1]{larson-24}.
	Consequently, the arguments above can also be adapted to obtain an alternate proof for \cite[Corollary 1.2]{larson-24}.
\end{rem}

\section*{References}
    \begingroup  
    \renewcommand{\section}[2]{}%
    \bibliographystyle{amsalpha}
    {\small\bibliography{ack.bib}}

\providecommand{\bysame}{\leavevmode\hbox to3em{\hrulefill}\thinspace}
\providecommand{\MR}{\relax\ifhmode\unskip\space\fi MR }
\providecommand{\MRhref}[2]{%
  \href{http://www.ams.org/mathscinet-getitem?mr=#1}{#2}
}
\providecommand{\href}[2]{#2}
\begin{thebibliography}{Sah25}

\bibitem[AB83]{atiyah-bott}
M.F. Atiyah and R.~Bott, \emph{The {Y}ang-{M}ills equations over {R}iemann surfaces}, Phil. Trans. R. Soc. Lond. A \textbf{308} (1983), 523--615.

\bibitem[CF07]{casnati}
G.~Casnati and C.~Fontanari, \emph{On the rationality of moduli spaces of pointed curves}, Journal of the London Mathematical Society \textbf{75} (2007), no.~3, 582--596.

\bibitem[dB01]{del-bano}
S.~del Ba{\~n}o, \emph{On the {C}how motive of some moduli spaces}, J Reine Angew Math \textbf{532} (2001), 105--132.

\bibitem[EK04]{kirwan-earl}
R.~Earl and F.~Kirwan, \emph{Complete sets of relations in the cohomology rings of moduli spaces of holomorphic bundles and parabolic bundles over a {R}iemann surface}, Proceedings of the London Mathematical Society \textbf{89} (2004), no.~3.

\bibitem[Fri18]{fringuelli}
R.~Fringuelli, \emph{The {P}icard group of the universal moduli space of vector bundles on stable curves}, Advances in Mathematics \textbf{336} (2018), 477–557.

\bibitem[GS]{shubham-higgs}
S.~Ghosh and S.~Saha, \emph{Cycles on the universal moduli of stable rank two bundles over hyperelliptic curves}, In preparation.

\bibitem[Hei10]{heinloth}
J.~Heinloth, \emph{Lectures on the {M}oduli {S}tack of {V}ector {B}undles on a {C}urve}, Affine Flag Manifolds and Principal Bundles (A.Schmitt, ed.), 2010.

\bibitem[HS14]{prada}
O.~Garcia-Prada{,}~J. Heinloth{,} and A.~Schmitt, \emph{On the motives of moduli of chains and {H}iggs bundles}, J. Eur. Math. Soc. \textbf{16} (2014), no.~12, 2617–2668.

\bibitem[Kir92]{kirwan}
F.~Kirwan, \emph{The {C}ohomology {R}ings of {M}oduli {S}paces of {B}undles over {R}iemann {S}urfaces}, Journal of the American Mathematical Society \textbf{5} (1992), no.~4.

\bibitem[Kir02]{kirwan-talk}
\bysame, \emph{Cohomology of {M}oduli {S}paces}, ICM Talk \textbf{1} (2002), 363--382.

\bibitem[KN98]{king-newstead}
A.D. King and P.E. Newstead, \emph{On the cohomology ring of the moduli space of rank 2 vector bundles on a curve}, Topology \textbf{37} (1998), no.~2, 407--418.

\bibitem[Lar23]{larson-18}
H.~Larson, \emph{The intersection theory of the moduli stack of vector bundles on $\mathbb {P}^1$}, Canadian Mathematical Bulletin \textbf{66} (2023), no.~2, 359--379.

\bibitem[Lar25]{larson-24}
\bysame, \emph{The {C}how ring of the universal {P}icard stack over the hyperelliptic locus}, Advances in Mathematics \textbf{479} (2025).

\bibitem[Lat22]{laterveer-franchetta}
R.~Laterveer, \emph{Algebraic cycles and intersections of three quadrics}, Math. Proc. Camb. Philos. Soc. \textbf{173} (2022), no.~2, 349--367.

\bibitem[Lie08]{lieblich}
M.~Lieblich, \emph{Twisted sheaves and the period-index problem}, Compos. Math \textbf{144} (2008), no.~1, 1--31.

\bibitem[PV15]{vakil}
N.~Penev and R.~Vakil, \emph{The {C}how {R}ing of the moduli space of curves of genus six}, Algebraic Geometry \textbf{2} (2015), no.~1, 123--136.

\bibitem[Sah25]{shubham-r2}
S.~Saha, \emph{Rational {C}how ring of the universal moduli space of semistable rank two bundles over genus two curves}, arxiv:2509.06764v2 (2025).

\bibitem[SY25]{bae-maulik}
Y.~Bae{,} D. Maulik{,}~J. Shen{,} and Q.~Yin, \emph{The intrinsic cohomology ring of the universal compactified {J}acobian over the moduli space of stable curves}, arXiv:2509.05577v2 (2025).

\bibitem[Zag95]{zagier}
D.~Zagier, \emph{On the {C}ohomology of {M}oduli {S}paces of {R}ank {T}wo {V}ector {B}undles {O}ver {C}urves}, The Moduli Space of Curves. Progress in Mathematics, 1995.

\end{thebibliography}
    \endgroup  
    
\appendix
\vspace{-0.1in}

\section{Macaulay2 scripts}
\begin{figure}[H]
\begin{lstlisting}[language=Macaulay2]
R = QQ[x,y,z,b, Degrees => {3:1,2}]; 
-- x= k_{0,1,0}, y = k_{-1,2,0}, z= k_{-1,0,1}, b= k_{-1,1,1}
H = 1/2*(4*z - y);T = 1/2*(3*x - y);
B = 1/2*(1/2*(y - x) - z)^2 - 1/2*(1/12*x*(6*y - 9*x) - b - 1/4*x^2 + 1/8*x^2+1/8*(4*z-y)^2);
-- using Cor 9 in Sah25, k_{0,0,1} = 1/4*k_{0,2,0}+1/2*(4*z-y)^2 
--k_{0,2,0}= 1/2*k_{0,1,0}^2 , k_{-1,3,0} = 3/4*k_{0,1,0}(2k_{-1,2,0}-3k_{0,1,0}) by Theorem 1.2
I = ideal(T^3,H^2*T^2 - 4*B*T^2,B*H^2 - B*H*T + 1/2*B*T^2,B*H^2 - B^2,H^3 + H^2*T + 1/2*H*T^2 - 4*B*H);
AU = R/I;  -- Chow ring of U(2,3,2)
s = reduceHilbert (hilbertSeries(AU));
cohom = QQ[a1,p1,p2,p3,a2,c2,Degrees=>{1:1,1:1,1:3,1:2,1:1}]/(p1^3, p1^2*p3, p1^2*p2+p1*p3^2, 6*p1*p2*p3+p3^3, p1*p3^3, p1*p2^2+p2*p3^2, p2^2*p3, p2*p3^3, p2^3 ,p3^5);
coh = reduceHilbert (hilbertSeries(cohom));
hilb = frac(QQ[t]);
s = (1+2*t+4*t^2+4*t^3+2*t^4+t^5)/(1-t);--Hilbert series of AU
r = t^2*(1+3*t+6*t^2+3*t^3+t^4)/((1-t^2)*(1-t)^2);-- Hilbert series of unstable locus
print(coh==s+r)--check conjecture
\end{lstlisting}
\caption{Macaulay2 Code: Conjecture for r,g=2}
\label{conj:code}
\end{figure}
\begin{figure}[H]
\begin{lstlisting}[language=Macaulay2]
rng = QQ[b11,b12,b13,b14,b21,b22,b23,b24, Degrees=>{4:1,4:3}, SkewCommutative=>true];
p1 = b11*b13 + b12*b14;
p2 = b21*b23 + b22*b24;
p3 = b11*b23 - b13*b21 + b12*b24 - b14*b22;
print(p1^3, p1^2*p3, p1^2*p2+p1*p3^2, 6*p1*p2*p3+p3^3, p1*p3^3, p1*p2^2+p2*p3^2, p2^2*p3, p2*p3^3, p2^3 ,p3^5)--verify relations
print(p3^2,p5,p3*p5,p4,p3*p4,p5^2,p3*p5^2,p4*p5,p4^2,p5^3,p4*p5^2,p5^4)--verify linear independence
\end{lstlisting}
\caption{Macaulay2 Code: The ring $A^*(\ve(2,d,2))$ - verifying relations}
\label{cohom:code}
\end{figure}
\end{document}